\documentclass[12pt]{article}

\usepackage{amssymb}%
\usepackage{amsmath}%
\usepackage{amsthm}%

\usepackage{fullpage}%
\usepackage{xcolor}%

\allowdisplaybreaks%

{\theoremstyle{plain}%
 \newtheorem{theorem}{Theorem}

 \newtheorem{lemma}{Lemma}%
}
{\theoremstyle{remark}

}
{\theoremstyle{definition}

}

\begin{document}

\begin{center}
{\Large On a conjecture on a series of convergence rate $\frac{1}{2}$}
\end{center}

\begin{center}
{\textsc{John M. Campbell} } 

 \ 

\end{center}

\begin{abstract}
 Sun, in 2022, introduced a conjectured evaluation for a series of convergence rate $\frac{1}{2}$ involving harmonic numbers. We prove both this 
 conjecture and a stronger version of this conjecture, using a summation technique based on a beta-type integral we had previously introduced. Our full 
 proof also requires applications of Bailey's 
 ${}_{2}F_{1}\left( \frac{1}{2} \right)$-formula, 
 Dixon's ${}_{3}F_{2}(1)$-formula, an almost-poised version of Dixon's formula due to Chu, Watson's formula for ${}_{3}F_{2}(1)$-series, 
 the Gauss summation theorem, Euler's formula for ${}_{2}F_{1}$-series, elliptic integral singular values, and lemniscate-like constants recently introduced 
 by Campbell and Chu. The techniques involved in our proof are useful, more broadly, in the reduction of difficult sums of convergence rate $\frac{1}{2}$ 
 to previously evaluable expressions. 
\end{abstract}



\section{Introduction}
 Zhi-Wei Sun has introduced many remarkable conjectures, over the years. Many of these conjectures are given by purported evaluations for very 
 difficult series that were discovered in an experimental fashion, with Computer Algebra System software and via numerical estimates. 
 Recently, Sun \cite{Sun2023} posted a preprint on series with summands involving harmonic numbers, and the purpose of our article is to prove one of 
 the conjectures given by Sun in this recent preprint \cite{Sun2023}. 

 Conjecture 2.4 from \cite{Sun2023} was formulated in the following manner by Sun \cite{Sun2023}, 
 and it was indicated \cite{Sun2023} that this Conjecture was introduced in December of 2022. 
 We are to let $H_{m} = 1 + \frac{1}{2} + \cdots + \frac{1}{m}$ denote the $m^{\text{th}}$ 
 entry in the sequence of harmonic numbers. 
 The $\Gamma$-function \cite[\S8]{Rainville1960} 
 is of great importance inside of mathematics and outside of mathematics 
 and is to be heavily involved in our article and may be defined via the Euler integral 
 so that $\Gamma(x)=\int_{0}^{\infty} u^{x-1}e^{-u}\,du$ for $\Re(x) > 0$. 

 \ 

\noindent {\bf Conjecture 2.4} from \cite{Sun2023}: We have 
\begin{equation}\label{Sun1}
 \sum_{k=0}^{\infty} \frac{ \binom{2k}{k}^{2} }{(-16)^{k}} (2 H_{2k} - H_{k}) 
 = -\frac{\ln(2) \, \Gamma^{2}\left( \frac{1}{4} \right)}{ 4 \pi \sqrt{2\pi} } 
\end{equation}
 and 
\begin{equation}\label{Sun2}
 \sum_{k=0}^{\infty} \frac{\binom{2k}{k}^{2}}{32^{k}} (2 H_{2k} - H_{k}) 
 = \frac{\ln(2) \, \Gamma^{2}\left( \frac{1}{4} \right)}{4 \pi \sqrt{\pi}}. 
\end{equation}
 
 \ 

 As it turns out, the alternating series evaluation in \eqref{Sun1} can be shown to follow in a direct way from results obtained via a hypergeometric 
 linearization method given by Chu and Campbell in \cite{ChuCampbell2021JMAA}, as we are to briefly demonstrate, later in this article. However, the 
 problem of proving the evaluation in \eqref{Sun2} 
 turns out to be much more difficult. We succeed in proving \eqref{Sun2}
 in this article, as in Section \ref{202390939939049594PM1A} below. 
 To the best of our knowledge, Sun's conjectured formula in \eqref{Sun2} 
 had not previously been proved. 

\subsection{Background}
 Using a linearization method based on coefficient extractions, Chu 
 \cite{Chu2022} recently applied this method to 
 obtain identities for evaluating series involving 
 $\frac{\binom{2k}{k}^2}{32^k}$ for $k \in \mathbb{N}_{0}$ together with harmonic-type numbers, 
 building on the beta integral-derived results previously given by Campbell in \cite{Campbell2019Rocky}. 
 However, the methds of Chu 
 \cite{Chu2022} cannot be applied, at least in any direct way, to prove \eqref{Sun2}. 
 In particular, the hypergeometric techniques due to Chu \cite{Chu2022} were applied in \cite{Chu2022} 
 to evaluate series of convergence rate $\frac{1}{2}$ involving 
\begin{equation}\label{2025350553350515130PM1A}
 \left( \frac{1}{32} \right)^{k} \binom{2k}{k}^{2} H_{k} 
\end{equation}
 for $k \in \mathbb{N}_{0}$, such as the formula 
\begin{equation}\label{2707283708838317172717477AM1A}
 \sum_{k=0}^{\infty} \left( \frac{1}{32} \right)^{k} 
 \binom{2k}{k}^2 \frac{H_{k}}{k+1} 
 = 8 - \frac{2\Gamma^{2}\left( \frac{1}{4} \right)}{\pi^{3/2}} - 
 \frac{4 \pi^{3/2} + 16 \sqrt{\pi} \ln(2)}{\Gamma^{2}\left( \frac{1}{4} \right)} 
\end{equation}
 previously introduced by Campbell \cite{Campbell2019Rocky}, 
 and series of convergence rate $\frac{1}{2}$ involving 
\begin{equation}\label{260762370637637071151PM1A}
 \left( \frac{1}{32} \right)^{k} \binom{2k}{k}^{2} O_{k}^{(2)}
\end{equation}
 for $k \in \mathbb{N}_{0}$, 
 writing $O_{k}^{(2)} = \sum_{i=1}^{k} \frac{1}{(2i-1)^2}$. 
 However, by expanding the summand of \eqref{Sun2}, 
 we would need to evaluate a series of convergence rate $\frac{1}{2}$ involving 
\begin{equation}\label{2888808283890989938939890989198185828P8M1A}
 \left( \frac{1}{32} \right)^{k} \binom{2k}{k}^{2} H_{2k} 
\end{equation}
 for $k \in \mathbb{N}_{0}$, in contrast to both \eqref{2025350553350515130PM1A} and \eqref{260762370637637071151PM1A}. 

 There have been only a few previously known series of convergence rate $\frac{1}{2}$ involving 
 \eqref{2888808283890989938939890989198185828P8M1A}, as in our previous work 
 on Fourier--Legendre theory and fractional calculus \cite{CampbellCantariniDAurizio2022}. 
 In particular, it was proved in \cite{CampbellCantariniDAurizio2022} that 
\begin{align}
 & \sum_{k = 0}^{\infty} \left( \frac{1}{32} \right)^{k} 
 \binom{2k}{k}^2 \frac{H_{2k}}{2k-1} = \label{202303311205AM1A} \\ 
 & \frac{\sqrt{\pi}\left( \pi + 3 \ln(2) - 4\right)}{2 \Gamma^{2}\left( \frac{1}{4} \right)} - 
 \frac{ \Gamma^{2}\left( \frac{1}{4} \right) \left( \pi - 3 \ln(2) - 2\right) }{16 \pi^{3/2}} \label{202303311205AM2A}
\end{align}
 and that 
\begin{equation}\label{7207277303371717205AM1A}
 \sum_{k = 0}^{\infty} 
 \left( \frac{1}{32} \right)^{k} \binom{2k}{k}^2 \frac{H_{2k}}{k+1} 
 = 4-\frac{3 \Gamma^2 \left(\frac{1}{4}\right)}{2 \pi ^{3/2}} - \frac{2 \sqrt{\pi } (\pi +3 \ln (2)-4)}{\Gamma^2 
 \left(\frac{1}{4}\right)}, 
\end{equation}
 and the results in \eqref{202303311205AM1A}--\eqref{202303311205AM2A} 
 and in \eqref{7207277303371717205AM1A} were highlighted as main results in \cite{CampbellCantariniDAurizio2022}. 
 Our approach toward solving the problem proposed by Sun given by proving 
 \eqref{Sun2} is of a similar nature relative to our past proofs of 
 \eqref{202303311205AM1A}--\eqref{202303311205AM2A} 
 and \eqref{7207277303371717205AM1A}, 
 but our proof of \eqref{Sun2} is considerably more involved 
 and requires recent results on what are referred to as \emph{lemniscate-like constants} in \cite{CampbellChu2021}. 

 As considered by Chu \cite{Chu2022}, the very fast convergence rate of series as in 
 \eqref{2707283708838317172717477AM1A} is of interest, especially in comparison with 
 previously known series of convergence rate $1$ involving $\frac{\binom{2k}{k}^2}{16^{k}}$ for $k \in \mathbb{N}_{0}$
 and harmonic numbers, as in the formula 
 $$ \sum_{k=0}^{\infty} \left( \frac{1}{16} \right)^{k} \binom{2k}{k}^2 \frac{H_{k}}{(2k-1)^2} 
 = \frac{12 - 16 \ln(2)}{\pi} $$ 
 proved by Choi in 2014 \cite{Choi2014} and independently by Chen in 2016 \cite{Chen2016}, which was later generalized by Campbell \cite{Campbell2018} 
 and by Wang and Chu \cite{WangChu2020}. 
 The foregoing considerations greatly motivate the development of new techniques for evaluating 
 series involving expressions as in \eqref{2025350553350515130PM1A}. 

\subsection{Preliminaries}
 One of the key tools that we are to apply to prove Sun's conjectured formula shown in \eqref{Sun2} is given by the beta-type integration technique 
 introduced in the author's previous publication \cite{Campbell2019Mediterr} and reproduced in the author's PhD Thesis \cite{Campbell2022PhD}. We are 
 to let $H_{m}' = 1 - \frac{1}{2} + \cdots + \frac{(-1)^{m+1}}{m}$ denote the $m^{\text{th}}$ entry in the sequence of alternating harmonic numbers, 
 recalling the relation such that $H_{2n}' = H_{2n} - H_{n}$. 

\begin{lemma}\label{mainlemH2np}
 Given a sequence $(f_{n} )_{n \geq 0}$ whereby the series 
\begin{equation}\label{20180919855PM1A}
 \sum_{n = 0 }^{\infty} 
 \left( \frac{1}{16} \right)^{n} H_{2n}' \binom{2n}{n}^{2} \frac{f_n}{n+1} 
\end{equation}
 converges, the above series is equal to 
\begin{align}
 & \frac{4}{\pi} \int_{0}^{1} 
 \sum_{n = 0}^{\infty} (-1)^{n} x^{2n} \sqrt{1-x^{2}}
 \binom{-\frac{1}{2}}{n} f_{n} \ln\left(x\right) \, dx \label{20220520911PM1A} \\ 
 & + \frac{1}{2} \sum _{n=0}^{\infty } 
 \left( \frac{1}{16} \right)^{n} \frac{ \binom{2 n}{n}^2 ( 2 \ln (2) ( n + 1) + 1)}{(n+1)^2} f_{n},
\end{align}
 under the assumption that the sequence $f$ is such that it is possible to reverse 
 the order of integration and infinite summation in \eqref{20220520911PM1A} \cite{Campbell2019Mediterr,Campbell2022PhD}. 
\end{lemma}

The evaluation of hypergeometric series is of central importance in our article, 
 so we find it to be appropriate to recall the following 
 definition of the term \emph{generalized hypergeometric series}: 
\begin{equation}\label{9209230939399191294858PM1A}
{}_{p}F_{q}\!\!\left[ 
\begin{matrix} 
a_{1}, a_{2}, \ldots, a_{p} \vspace{1mm}\\ 
b_{1}, b_{2}, \ldots, b_{q} 
\end{matrix} \ \Bigg| \ x \right] 
= \sum_{n=0}^{\infty} 
\left[ \begin{matrix} a_{1}, a_{2}, \ldots, a_{p} \vspace{1mm} \\ 
b_{1}, b_{2}, \ldots, b_{q} \end{matrix} \right]_{n} \frac{x^{n}}{n!}. 
\end{equation}
 The \emph{complete elliptic integrals} are to play a significant role in our proof of \eqref{Sun2}. 
 For the purposes of our article, the complete elliptic integrals $\text{{\bf K}}$ and $\text{{\bf E}}$
 of the first and second kinds may be defined via the following Maclaurin series expansions, 
 with reference to the classic \emph{Pi and the AGM} text \cite[pp.\ 8--10]{BorweinBorwein1987}: 
\begin{align}
 & \text{{\bf K}}(k) = 
 \frac{\pi}{2} \cdot {}_{2}F_{1}\!\!\left[ 
 \begin{matrix}
 \frac{1}{2}, \frac{1}{2} \vspace{1mm}\\
 1
 \end{matrix} \ \Bigg| \ k^2 \right], \label{20210716754PM1A} \\
 & \text{{\bf E}}(k) = \frac{\pi}{2} \cdot {}_{2}F_{1}\!\!\left[ 
 \begin{matrix}
 \frac{1}{2}, -\frac{1}{2} \vspace{1mm}\\ 1
 \end{matrix} \ \Bigg| \ k^2 \right]. \label{20210716755PM2A}
\end{align}
 The following elliptic integral singular values highlighted as Theorem 1.7 in the \emph{Pi and the AGM} text 
 \cite[p.\ 25]{BorweinBorwein1987} are to also be applied in our proof 
 of \eqref{Sun2}: 
\begin{align} 
 & \text{{\bf K}}\left( \frac{1}{\sqrt{2}} \right) = \frac{\Gamma^{2}\left( \frac{1}{4} 
 \right)}{4\sqrt{\pi}}, \label{20239093390900PM1A} \\
 & \text{{\bf E}}\left( \frac{1}{\sqrt{2}} \right) = 
 \frac{\Gamma^2 \left(\frac{1}{4}\right)}{8 \sqrt{\pi }} + 
 \frac{\pi ^{3/2}}{\Gamma^2 \left(\frac{1}{4}\right)}. \label{20239093390900PM2A} 
\end{align}
 Writing $k' = \sqrt{1 - k^2}$, we are to also employ 
 the following differential equations \cite[p.\ 10]{BorweinBorwein1987}: 
\begin{equation}\label{7207273707373707875747P7M1A}
 \frac{d\text{{\bf E}}}{dk} 
 = \frac{ \text{{\bf E}} - \text{{\bf K}} }{k} 
 \ \ \ \text{and} \ \ \ 
 \frac{d\text{{\bf K}}}{dk} = \frac{\text{{\bf E}} - (k')^2 \text{{\bf K}}}{k (k')^2}. 
\end{equation}
 We are to make use of the \emph{Gauss's summation theorem} \cite[\S1.3]{Bailey1935} such that
\begin{equation}\label{2023033017171757P77M1A}
 {}_{2}F_{1}\!\!\left[ 
 \begin{matrix} 
 a, b \vspace{1mm}\\ c 
 \end{matrix} \ \Bigg| \ 1 \right] 
 = \Gamma \left[ \begin{matrix} c, c - a - b \vspace{1mm} \\ 
 c-a, c - b \end{matrix} \right] 
\end{equation}
 for $\Re(c - a - b) > 0$, writing 
\begin{equation*}
 \Gamma \left[ \begin{matrix} \alpha, \beta, \ldots, \gamma \vspace{1mm} \\ 
 A, B, \ldots, C \end{matrix} \right] = 
 \frac{ \Gamma(\alpha) \Gamma(\beta) 
 \cdots \Gamma(\gamma) }{ \Gamma(A) \Gamma(B) \cdots \Gamma(C)}. 
\end{equation*}
 
 The \emph{Euler--Mascheroni constant} is such that $ \gamma = \lim_{n\to\infty}\big(H_n-\ln n\big)$. The \emph{digamma function} is 
 the special function such that \cite[\S9]{Rainville1960}: 
\begin{equation}\label{psidef}
 \psi(z) = \frac{d}{dz}\ln\Gamma(z)=\frac{\Gamma'(z)}{\Gamma(z)} =-\gamma+\sum_{n=0}^{\infty}\frac{z-1}{(n+1)(n+z)}. 
\end{equation}

 Euler's formula for ${}_{2}F_{1}$-series \cite[p.\ 57]{Luke1969} is such that 
\begin{align*}
 & {}_{2}F_{1}\!\!\left[ 
 \begin{matrix} 
 a, b \vspace{1mm} \\ 
 c 
 \end{matrix} \ \Bigg| \ z \right] = \Gamma \left[ \begin{matrix} c \vspace{1mm} \\ 
 c-a, a \end{matrix} \right] 
 \int_{0}^{1} t^{a-1} (1-t)^{c-a-1} (1 - z t)^{-b} \, dt, 
\end{align*}
 for $\Re(c) > \Re(a) > 0$ and for $|\text{arg}(1-z)| < \pi$. 
 Following \cite[p.\ 111]{Luke1969}, Euler's formula may be applied 
 to obtain hypergeometric reduction formula 
\begin{equation}\label{transform3F22F1}
 z {}_{3}F_{2}\!\!\left[ 
 \begin{matrix}
 a, b, 1 \vspace{1mm} \\ 
 c, 2 
 \end{matrix} \ \Bigg| \ z \right] 
 = \frac{c-1}{(a-1)(b-1)} 
 \left( {}_{2}F_{1}\!\!\left[ 
 \begin{matrix} 
 a - 1, b - 1 \vspace{1mm} \\ 
 c - 1 
 \end{matrix} \ \Bigg| \ z \right] - 1\right)
\end{equation}
 given in \cite[p.\ 111]{Luke1969}. So, 
 the Gauss identity in \eqref{2023033017171757P77M1A} gives us, from \eqref{transform3F22F1}, that 
\begin{equation}\label{202388038380818182717P7M1A}
 {}_{3}F_{2}\!\!\left[ 
 \begin{matrix} 
 a, b, 1 \vspace{1mm} \\ 
 c, 2 
 \end{matrix} \ \Bigg| \ 1 \right] 
 = \frac{c-1}{(a-1)(b-1)} 
 \left( \Gamma \left[ \begin{matrix} c-1, c-a-b+1 \vspace{1mm} \\ 
 c-a, c-b \end{matrix} \right] - 1\right) 
\end{equation}
 for $a \neq 1$, $b \neq 1$, and $\Re(c-a-b) > -1$ \cite[p.\ 111]{Luke1969}. 
 By L'H\^{o}pital's rule, we may obtain from \eqref{202388038380818182717P7M1A} that 
\begin{equation}\label{3F2digamma}
 {}_{3}F_{2}\!\!\left[ 
 \begin{matrix} 
 a, 1, 1 \vspace{1mm} \\ 
 c, 2 
 \end{matrix} \ \Bigg| \ 1 \right] 
 = \frac{c-1}{a-1} \left( \psi(c-1) - \psi(c-a) \right)
\end{equation}
 for $a \neq 1$ and $\Re(c-a) > 0$ \cite[p.\ 111]{Luke1969}. 
 The identity in \eqref{3F2digamma}
 is to be required in a direct way in our proof of \eqref{Sun2}. 

 The evaluation 
\begin{equation}\label{3F2CCD}
 {}_{3}F_{2}\!\!\left[ 
 \begin{matrix} 
 \frac{1}{2}, 1, \frac{5}{4} \vspace{1mm} \\ 
 \frac{3}{2}, \frac{7}{4} 
 \end{matrix} \ \Bigg| \ 1 \right] = 3 - \frac{6\pi^3}{\Gamma^{4}\left( \frac{1}{4} \right)} 
\end{equation}
 was introduced by Campbell et al.\ in \cite{CampbellCantariniDAurizio2022}
 and highlighted as part of a main Theorem from \cite{CampbellCantariniDAurizio2022}. 
 This evaluation in \eqref{3F2CCD} is to provide a key tool 
 in our proof of Sun's conjectured formula in \eqref{Sun2}. 
 The ${}_{3}F_{2}(1)$-evaluation in \eqref{3F2CCD} 
 was proved in \cite{CampbellCantariniDAurizio2022} using FL theory and fractional operators. 
 The evaluation in \eqref{3F2CCD} was applied in a similar context in 
 \cite{CampbellChu2021}, in which \eqref{3F2CCD} was proved using
 a reindexing argument together with Watson's formula for ${}_{3}F_{2}(1)$-series \cite[\S1.3]{Bailey1935}: 
\begin{equation}\label{99920293909331929295PM1A}
 {}_{3}F_{2}\!\!\left[ 
 \begin{matrix} 
 a, b, c \vspace{1mm} \\ 
 \frac{a + b + 1}{2}, 2 c 
 \end{matrix} \ \Bigg| \ 1 \right] = 
 \Gamma \left[ \begin{matrix} \frac{1}{2}, c + \frac{1}{2}, \frac{a+b+1}{2}, c - \frac{a + b - 1}{2} \vspace{1mm} \\ 
 \frac{a+1}{2}, \frac{b+1}{2}, c - \frac{a-1}{2}, c - \frac{b-1}{2} \end{matrix} \right]. 
\end{equation}
 Indeed, by following \cite{CampbellChu2021}, we may apply a reindexing argument to the series in 
 \eqref{3F2CCD} so that 
\begin{align*}
 {}_{3}F_{2}\!\!\left[ 
 \begin{matrix} 
 \frac{1}{2}, 1, \frac{5}{4} \vspace{1mm} \\ 
 \frac{3}{2}, \frac{7}{4} 
 \end{matrix} \ \Bigg| \ 1 \right] 
 & = \sum_{k=0}^{\infty} \frac{ \left( \frac{1}{2} \right)_{k} 
 \left( \frac{5}{4} \right)_{k} }{ \left( \frac{3}{2} \right)_{k} \left( \frac{7}{4} \right)_{k} } \\ 
 & = -3 \sum_{k = 1}^{\infty} \frac{ \left( -\frac{1}{2} \right)_{k} 
 \left( \frac{1}{4} \right)_{k} }{ \left( \frac{1}{2} \right)_{k} \left( \frac{3}{4} \right)_{k} } \\ 
 & = 3 \left( 1 - {}_{3}F_{2}\!\!\left[ 
 \begin{matrix} 
 1, -\frac{1}{2}, \frac{1}{4} \vspace{1mm} \\ 
 \frac{1}{2}, \frac{3}{4} 
 \end{matrix} \ \Bigg| \ 1 \right] \right), 
\end{align*}
 and so that the formulation of Watson's identity in \eqref{99920293909331929295PM1A}
 then gives us the desired symbolic form in \eqref{3F2CCD}. 

 We let $$ \beta(x, y) = \int_0^1 t^{x-1} (1-t)^{y-1} \, dt $$ 
 denote the beta function for $\Re(x) > 0$ and $\Re(y) > 0$. 
 The classical \emph{lemniscate constants} 
\begin{equation}\label{lemniscateA}
 A = \frac{1}{4} \beta\left(\frac{1}{2},\frac{1}{4}\right) = \int_0^1 \frac{1}{\sqrt{1-t^4}} \, dt = 
 \sum _{n=0}^{\infty } \left(\frac{1}{4}\right)^n \binom{2 n}{n} \frac{ 1 }{4 n+1} = 
 \frac{\Gamma^2 \left(\frac{1}{4}\right)}{4 \sqrt{2 \pi }} 
\end{equation}
 and 
\begin{equation}\label{lemniscateB}
 B = \frac{1}{4} \beta\left(\frac{1}{2},\frac{3}{4}\right) = 
 \int_0^1 \frac{t^2}{\sqrt{1-t^4}} \, dt 
 = \sum _{n=0}^{\infty } \left(\frac{1}{4}\right)^n \binom{2 n}{n} \frac{1}{4 n+3} 
 = \frac{\sqrt{2 \pi ^3}}{\Gamma^2 \left(\frac{1}{4}\right)} 
\end{equation}
 have been of much significance in the history of mathematics \cite{CampbellChu2021,Todd1975}, 
 and this led to the exploration of \emph{lemniscate-like constants} of the following forms \cite{Campbell2021Integers,CampbellChu2021}, for a suitable 
 sequence $f$: 
\begin{equation}\label{20230331401PM1A}
 \sum _{n=0}^{\infty } \left(\frac{1}{4}\right)^n \binom{2 n}{n} \frac{ f_{n} }{4 n+1} 
 \ \ \ \text{and} \ \ \ 
 \sum _{n=0}^{\infty } \left(\frac{1}{4}\right)^n \binom{2 n}{n} \frac{ f_{n} }{4 n+3}. 
\end{equation}
 Series as in \eqref{20230331401PM1A} are to be heavily used in our main proof in Section \ref{202390939939049594PM1A}. 

 \emph{Dixon's formula} for well-poised series is such that 
\begin{equation}\label{Dixonmain}
 {}_{3}F_{2}\!\!\left[ 
 \begin{matrix} 
 a, b, c \vspace{1mm} \\ 
 1 + a - b, 1 + a - c 
 \end{matrix} \ \Bigg| \ 1 \right] = 
 \Gamma \left[ \begin{matrix} 1 + \frac{a}{2}, 1 + \frac{a}{2} - b - c, 
 1 + a -b, 1 + a - c \vspace{1mm} \\ 
 1 + a, 1 + a - b - c, 1 + \frac{a}{2} - b, 1 + \frac{a}{2} - c \end{matrix} \right]. 
\end{equation}
 Chu \cite{Chu2012WWD} introduced extended versions of {Watson}--{Whipple}--{Dixon} {{\(_{3}F_{2}\)}}-series, and we are to apply the following 
 almost-poised version of Dixon's formula \cite{Chu2012WWD} as in \cite{CampbellChu2021}: 
\begin{align*}
 & {}_{3}F_{2}\!\!\left[ \begin{matrix} a, b, c \vspace{1mm} \\ 
 2 + a - b, 2 + a - c \end{matrix} \ \Bigg| \ 1 \right] = \\ 
 & \frac{2^{1 + 2a - 2 b - 2 c} \Gamma(a - b + 2) \Gamma(a - c + 2)}{\pi (b - 1) (1 - c) \Gamma(a) \Gamma(a - 2 b +2) \Gamma(a - 2 c + 2) 
 \Gamma(a - b - c + 2)} \\ 
 & \Bigg( \Gamma\left( \frac{1 + a}{2} \right) \Gamma\left( \frac{2 + a }{2} - b \right) \Gamma\left( \frac{2 + a}{2} - c \right) \Gamma\left( \frac{5 + 
 a}{2} - b - c \right) \\ 
 & - \Gamma\left( \frac{a}{2} \right) \Gamma\left( \frac{3 + a}{2} - b \right) \Gamma\left( \frac{3 + a}{2} - c \right) \Gamma\left( \frac{4 + 
 a}{2} - b - c \right) \Bigg). 
\end{align*}
 A direct application of the above hypergeometric identity gives us that: 
\begin{equation}\label{directalmostpoised}
 \sum_{n=0}^{\infty} 
 \left( \frac{1}{4} \right)^{n} 
 \binom{2n}{n} \frac{1}{(4n+3)^2} = 
 \frac{1}{9} \, 
 {}_{3}F_{2}\!\!\left[ 
 \begin{matrix} 
 \frac{1}{2}, \frac{3}{4}, \frac{3}{4} \vspace{1mm} \\ 
 \frac{7}{4}, \frac{7}{4} 
 \end{matrix} \ \Bigg| \ 1 \right] = 
 \frac{4 - \pi}{4 \sqrt{2 \pi}} \Gamma^{2}\left( \frac{3}{4} \right). 
\end{equation}
 The lemniscate-like constant evaluation such that 
\begin{equation}\label{Slovaca2}
 \sum_{n=0}^{\infty} \left( \frac{1}{4} \right)^{n} 
 \binom{2n}{n} \frac{O_{2n}}{4n+3} 
 = \frac{\pi^{3/2}(3 \ln(2) + 2)}{2 \sqrt{2} \Gamma^{2}\left( \frac{1}{4} \right)} 
\end{equation}
 and 
\begin{equation}\label{Slovaca1}
 \sum_{k=0}^{\infty} \left( \frac{1}{4} \right)^{k} \binom{2k}{k} 
 \frac{O_{2k}}{4k+1} = \frac{3 \Gamma^{2}\left( \frac{1}{4} \right) \ln(2) }{16 \sqrt{2\pi}} 
\end{equation}
 were proved in \cite{CampbellChu2021} and included as main results, 
 and we are to apply \eqref{Slovaca2} and \eqref{Slovaca1}
 in our main proof, letting $O_{m} = 1 + \frac{1}{3} + \cdots + \frac{1}{2m-1}$ denote the $m^{\text{th}}$ odd harmonic number. 

 \emph{Bailey's theorem} \cite[p.\ 11]{Bailey1935} is such that 
\begin{equation}\label{20283880383817881707P7M1A}
 {}_{2}F_{1}\!\!\left[ 
 \begin{matrix} 
 a, 1-a \vspace{1mm} \\ c 
 \end{matrix} \ \Bigg| \ \frac{1}{2} \right] = 
 \Gamma \left[ \begin{matrix} \frac{c}{2}, \frac{c+1}{2} \vspace{1mm} \\ 
 \frac{a+c}{2}, \frac{1 - a + c}{2} \end{matrix} \right]. 
\end{equation}
 Tauraso \cite{Tauraso2018} obtained the following, by setting $a = \frac{1}{2}$ in 
 \eqref{20283880383817881707P7M1A} and via a term-by-term application 
 of the operator $\frac{\partial}{\partial c} \cdot \big|_{c=1}$:
\begin{equation}\label{mainTauraso}
 \sum_{k = 0}^{\infty} \left( \frac{1}{32} \right)^{k} \binom{2k}{k}^2 H_{k} 
 = \frac{\sqrt{\pi}(\pi - 4 \ln(2))}{2 \Gamma^{2}\left( \frac{3}{4} \right)}. 
\end{equation}
 However, the application of operators such as $\frac{\partial}{\partial c} \cdot \big|_{c=1/2}$ have the effect of reducing the power of central binomial 
 coefficients, so it is unclear as to how it may be possible to mimic Tauraso's approach in the hope of evaluating the following intractable series: $$ 
 \sum_{k=0}^{\infty} \left( \frac{1}{32} \right)^{k} \binom{2k}{k}^2 H_{2k}. $$ 

\section{Proof of a conjectured evaluation for   a   series of convergence rate $\frac{1}{2}$}\label{202390939939049594PM1A}

\begin{theorem}
 The formula in \eqref{Sun2} conjectured by Sun holds true. 
\end{theorem}

\begin{proof}
 We begin by setting $f_{n} = 2^{-n} (n+1)$ in Lemma \ref{mainlemH2np}. Lemma \ref{mainlemH2np} then gives us the equality of 
\begin{equation}\label{maindesired}
 \sum_{n = 
 0}^{\infty} \left(\frac{1}{32}\right)^n \left(H_{2 n}-H_n\right) \binom{2 n}{n}^2 
\end{equation}
 and 
\begin{align*}
 & \frac{4}{\pi } \int_0^1 \sqrt{1-x^2} \ln (x) \sum _{n = 
 0}^{\infty} \left(-\frac{x^2}{2}\right)^n \binom{-\frac{1}{2}}{n} (n+1) \, dx + \\ 
 & \frac{1}{2} \sum_{n = 0}^{\infty} 
 \left(\frac{1}{32}\right)^n \binom{2 n}{n}^2
 \frac{ 1+2 (n + 1) \ln (2)}{n+1}. 
\end{align*}
 According to the Maclaurin series expansions in \eqref{20210716754PM1A} and \eqref{20210716755PM2A} 
 along with the differential equations in \eqref{7207273707373707875747P7M1A}, 
 we may obtain the power series expansion 
\begin{equation}\label{9928808230383808885898P8M1A}
 \sum _{n=0}^{\infty } \frac{\binom{2 n}{n}^2 }{n+1} y^n 
 = \frac{\text{{\bf E}}\left(4 \sqrt{y}\right)}{4 \pi y} + 
 \frac{4 \left(1-\frac{1}{16 y}\right) \text{{\bf K}}\left(4 \sqrt{y}\right)}{\pi}. 
\end{equation}
 From the expansion in \eqref{9928808230383808885898P8M1A} together with the elliptic integral 
 singular values shown in \eqref{20239093390900PM1A} and in \eqref{20239093390900PM2A}, 
 we find that the series in \eqref{maindesired} is expressible in the following manner: 
\begin{align}
 & \frac{4}{\pi } \int_0^1 \sqrt{1-x^2} \ln (x) \sum _{n = 
 0}^{\infty} \left(-\frac{x^2}{2}\right)^n \binom{-\frac{1}{2}}{n} (n+1) \, dx + \label{2023033100000000911PM1A} \\ 
 & \frac{4 \sqrt{\pi }}{\Gamma^2 \left(\frac{1}{4}\right)} + 
 \frac{\ln (2) \Gamma^2 \left(\frac{1}{4}\right)}{2 \pi ^{3/2}}. \label{2023033100000000911PM2A} 
\end{align}
 By the generalized binomial theorem, we find that 
 \eqref{2023033100000000911PM1A}--\eqref{2023033100000000911PM2A}
 is reducible to the following: 
\begin{equation}\label{bygenbin}
 -\frac{4}{\pi \sqrt{2}} \int_0^1 \frac{\sqrt{1-x^2} \left(x^2 - 4\right) \ln (x)}{\left(2-x^2\right)^{3/2}} \, dx 
 + \frac{4 \sqrt{\pi }}{\Gamma^2 \left(\frac{1}{4}\right)} + 
 \frac{\Gamma^2 \left(\frac{1}{4}\right) \ln (2)}{2 \pi ^{3/2}}. 
\end{equation}
 So, it remains to evaluate the integral in \eqref{bygenbin}, i.e., to evaluate the following expression: 
\begin{equation}\label{inttosplit}
 \int_0^1 \frac{\sqrt{1-x^2} \left(x^2 - 4\right) \ln (x)}{\left(2-x^2\right)^{3/2}} \, dx. 
\end{equation}
 We may rewrite \eqref{inttosplit} as 
\begin{equation}\label{rewritetosplit}
 -\int_0^1 \sqrt{\frac{1-x^2}{2-x^2}} \ln (x) \, dx-2 \int_0^1 \frac{\sqrt{1-x^2} \ln (x)}{\left(2-x^2\right)^{3/2}} \, dx. 
\end{equation}
 Applying the change of variables such that $1 - x^2 = u$ to the first integral in 
 \eqref{rewritetosplit}, we obtain 
\begin{equation}\label{afterchange}
 -\frac{1}{4} \int_0^1 \frac{\sqrt{u} \ln (1-u)}{\sqrt{1-u^2}} \, du 
 - 2 \int_0^1 \frac{\sqrt{1-x^2} \ln (x)}{\left(2-x^2\right)^{3/2}} \, dx. 
\end{equation}
   By expanding the integrand factor $\ln(1-u)$ with its Maclaurin series and then integrating term-by-term   using the Dominated Convergence Theorem,    
  we may obtain the following from \eqref{afterchange}:   
\begin{equation}\label{202370737377079475PM1A}
 \frac{\sqrt{\pi}}{8} 
 \sum _{n = 1}^{\infty } \frac{1}{n } 
 \Gamma \left[ \begin{matrix} \frac{n}{2}+\frac{3}{4} \vspace{1mm} \\ 
 \frac{n}{2}+\frac{5}{4} \end{matrix} \right] - 2 
 \int_0^1 \frac{\sqrt{1-x^2} \ln (x)}{\left(2-x^2\right)^{3/2}} \, dx. 
\end{equation}
 Applying a series bisection to \eqref{202370737377079475PM1A}, we obtain 
\begin{align}
 & \frac{\sqrt{\pi }}{16} \sum _{n=1}^{\infty} \frac{1}{n} 
 \Gamma \left[ \begin{matrix} n+\frac{3}{4} \vspace{1mm} \\ 
 n+\frac{5}{4} \end{matrix} \right] 
 + \frac{\sqrt{\pi }}{8} \sum _{n=1}^{\infty} \frac{1}{2 n - 1 } \Gamma \left[ \begin{matrix} n + \frac{1}{4} \vspace{1mm} \\ 
 n + \frac{3}{4} \end{matrix} \right] - \label{20230330180827797P7M1A} \\ 
 & 2 \int_0^1 \frac{\sqrt{1-x^2} \ln (x)}{\left(2-x^2\right)^{3/2}} \, dx. \nonumber 
\end{align}
 Applying an index shift to the first series in \eqref{20230330180827797P7M1A}, we obtain: 
\begin{equation}\label{shiftto3F2}
 \sum _{n = 1}^{\infty } \frac{1}{n } 
 \Gamma \left[ \begin{matrix} n + \frac{3}{4} \vspace{1mm} \\ 
 n + \frac{5}{4} \end{matrix} \right] 
 = \frac{12 \Gamma \left(\frac{3}{4}\right) }{5 \Gamma \left(\frac{1}{4}\right)} 
 {}_{3}F_{2}\!\!\left[ 
 \begin{matrix} 
 1, 1, \frac{7}{4} \vspace{1mm}\\ 
 2, \frac{9}{4} 
 \end{matrix} \ \Bigg| \ 1 \right]. 
\end{equation}
 So, the digamma identity in \eqref{3F2digamma} from \cite[p.\ 111]{Luke1969} allows us to evaluate the ${}_{3}F_{2}$-expression in \eqref{shiftto3F2}. 
 This allows us to rewrite \eqref{afterchange} in the following manner: 
\begin{align}
 & -\frac{\pi ^{3/2} (\pi +2 \ln (2) -8)}{4 \sqrt{2} \Gamma^2 \left(\frac{1}{4}\right)} + \frac{\sqrt{\pi }}{8} \sum _{n = 
 1}^{\infty} \frac{1}{2 n-1} 
 \Gamma \left[ \begin{matrix} n+\frac{1}{4} \vspace{1mm} \\ 
 n+\frac{3}{4} \end{matrix} \right] - \label{rewriteafterchange1} \\ 
 & 2 \int_0^1 \frac{\sqrt{1-x^2} \ln (x)}{\left(2-x^2\right)^{3/2}} \, dx. \label{rewriteafterchange2} 
\end{align}
 By rewriting the infinite series in \eqref{rewriteafterchange1} according to 
 the notation in \eqref{9209230939399191294858PM1A}, 
 we find that \eqref{rewriteafterchange1}--\eqref{rewriteafterchange2} is equal to: 
\begin{align*}
 & -\frac{\pi ^{3/2} (\pi +2 \ln (2) - 8)}{4 \sqrt{2} \Gamma^2 \left(\frac{1}{4}\right)} + 
 \frac{\Gamma^2 \left(\frac{1}{4}\right)}{24 \sqrt{2 \pi }} 
 {}_{3}F_{2}\!\!\left[ 
 \begin{matrix} 
 \frac{1}{2}, 1, \frac{5}{4} \vspace{1mm} \\ 
 \frac{3}{2}, \frac{7}{4} 
 \end{matrix} \ \Bigg| \ 1 \right] - \\
 & 2 \int_0^1 \frac{\sqrt{1-x^2} \ln (x)}{\left(2-x^2\right)^{3/2}} \, dx. 
\end{align*}
 So, from the Watson-derived ${}_{3}F_{2}(1)$-evaluation on display in \eqref{3F2CCD}, we find that the integral in \eqref{inttosplit} is equal to 
 the following: 
\begin{equation}\label{oneintremains}
 \frac{\Gamma^2 \left(\frac{1}{4}\right)}{8 \sqrt{2 \pi }} - 
 \frac{\pi ^{3/2} (\pi +\ln (2)-4)}{2 \sqrt{2} \Gamma^2 \left(\frac{1}{4}\right)}-2 \int_0^1 
 \frac{\sqrt{1-x^2} \ln (x)}{\left(2-x^2\right)^{3/2}} \, dx. 
\end{equation}
 So, it remains to evaluate the integral in \eqref{oneintremains}. Using a change of variables, we obtain that: 
 $$ \int_0^1 \frac{\sqrt{1-x^2} \ln (x)}{\left(2-x^2\right)^{3/2}} \, dx
 = \frac{1}{4} \int_0^1 \frac{\sqrt{u} \ln (1-u)}{(1+u)^{3/2} \sqrt{1-u}} \, du. $$
 Using an appropriate Cauchy product, we may obtain that 
\begin{equation}\label{useCauchy}
 \left( \left( \frac{d}{du} \right)^{n} \frac{1}{(u+1) \sqrt{1-u^2}} \right) \, \Bigg|_{u=0} 
 = \left(-\frac{1}{2}\right)^n (n+1)! \binom{n}{\left\lfloor \frac{n}{2}\right\rfloor }. 
\end{equation}
 So, from the Maclaurin series corresponding to \eqref{useCauchy}, we may obtain that 
\begin{align}
 & \int_0^1 \frac{\sqrt{1-x^2} \ln (x)}{\left(2-x^2\right)^{3/2}} \, dx = \label{applymoment1} \\ 
 & \frac{1}{4} \int_0^1 \left(\sum _{n = 
 0}^{\infty} \left(-\frac{1}{2}\right)^n \binom{n}{\left\lfloor \frac{n}{2}\right\rfloor } (n+1) u^{n+\frac{1}{2}}
 \ln (1-u)\right) \, du. \label{applymoment2}
\end{align}
 According to a standard moment formula for the digamma function, we have that:
\begin{equation}\label{momentpsi}
 \int_0^1 u^{n+\frac{1}{2}} \ln (1-u) \, du 
 = -\frac{2 \left(\psi\left(n+\frac{5}{2}\right)+\gamma \right)}{2 n+3}. 
\end{equation}
 According to the expansion formula for the digamma function shown in \eqref{psidef}, we may obtain from \eqref{momentpsi} that 
\begin{equation}\label{momentfinite}
 \int_0^1 u^{n+\frac{1}{2}} \ln (1-u) \, du 
 = \frac{4 \ln (2) - 4 O_{n+1} - \frac{4}{2 n+3}}{2 n+3}. 
\end{equation}
 According to the Dominated Convergence Theorem, we may reverse the order of integration and infinite summation with respect to the equality in 
 \eqref{applymoment1}--\eqref{applymoment2}, 
 so as to give us the following, being consistent with the notation in \eqref{momentfinite}: 
\begin{align}
 & \int_0^1 \frac{\sqrt{1-x^2} \ln (x)}{\left(2-x^2\right)^{3/2}} \, dx = \nonumber \\ 
 & \frac{1}{4} \sum _{n = 
 0}^{\infty } \left(-\frac{1}{2}\right)^n \binom{n}{\left\lfloor \frac{n}{2}\right\rfloor }
 \frac{ (n+1) \left(4 
 \ln (2)-4 O_{n+1} -\frac{4}{2 n+3}\right)}{2 n+3}. \nonumber
\end{align}
 Applying a series bisection, we obtain that 
\begin{align}
 & \int_0^1 \frac{\sqrt{1-x^2} \ln (x)}{\left(2-x^2\right)^{3/2}} \, dx = \nonumber \\ 
 & \frac{1}{4} \sum_{n = 
 0}^{\infty } \left(\frac{1}{4}\right)^n \binom{2 n}{n}
 \frac{ (2 n+1) \left(4 \ln (2) - 4 
 O_{2 n+1}-\frac{4}{2 (2 
 n)+3}\right)}{2 (2 n)+3} - \\ 
 & \frac{1}{4} \sum _{n = 
 0}^{\infty} \left(\frac{1}{2}\right)^{2 n+1} \binom{2 n+1}{n}
 \frac{ ((2 n+1)+1) \left(4 \ln (2)-4 
 O_{(2 n+1)+1} - 
 \frac{4}{2 (2 n+1)+3}\right)}{2 (2 n+1)+3}. \nonumber 
 \end{align}
 Equivalently, 
\begin{align*}
 & \int_0^1 \frac{\sqrt{1-x^2} \ln (x)}{\left(2-x^2\right)^{3/2}} \, dx = \\
 & -\frac{1}{8} \sum_{n = 
 0}^{\infty} \left(\frac{1}{4}\right)^n \binom{2 n}{n}
 \frac{1}{4 n+1}+\left(-1-\frac{\ln (2)}{2}\right) \sum _{n=0}^{\infty} 
 \left(\frac{1}{4}\right)^n \binom{2 n}{n} \frac{ 1 }{4 n+3} + \\ 
 & \frac{13+12 \ln (2)}{8} \sum _{n=0}^{\infty} 
 \left(\frac{1}{4}\right)^n \binom{2
 n}{n} \frac{1}{4 n + 
 5}+\frac{1}{2} \sum _{n=0}^{\infty} \left(\frac{1}{4}\right)^n \binom{2 n}{n} 
 \frac{1}{(4 n+3)^2} - \\ 
 & \frac{3}{2} \sum_{n=0}^{\infty} 
 \left(\frac{1}{4}\right)^n \binom{2 n}{n} 
 \frac{1}{(4 n+5)^2} + 
 \frac{1}{2} \sum_{n=0}^{\infty} 
\left(\frac{1}{4}\right)^n \binom{2 n}{n} 
 \frac{ O_{2 n}}{4 n+3} - \\
 & \frac{3}{2} \sum_{n=0}^{\infty} \left(\frac{1}{4}\right)^n \binom{2 n}{n}
 \frac{ O_{2 n}}{4 n+5}. 
\end{align*}
 From \eqref{lemniscateA} and \eqref{lemniscateB}, we may obtain that 
\begin{align*}
 & \int_0^1 \frac{\sqrt{1-x^2} \ln (x)}{\left(2-x^2\right)^{3/2}} \, dx = \\ 
 & -\frac{\Gamma^2 \left(\frac{1}{4}\right)}{32 \sqrt{2 \pi }}+\frac{\left(-1-\frac{\ln (2)}{2}\right) \sqrt{2 \pi ^3}}{\Gamma^2 
 \left(\frac{1}{4}\right)} + \\
 & \frac{13+12 \ln (2)}{8} \sum_{n=0}^{\infty} \left(\frac{1}{4}\right)^n \binom{2 n}{n}
 \frac{1}{4 n+5} + \frac{1}{2} \sum_{n = 
 0}^{\infty} \left(\frac{1}{4}\right)^n \binom{2 n}{n} \frac{1}{(4 n+3)^2} - \\ 
 & \frac{3}{2} \sum _{n=0}^{\infty} \left(\frac{1}{4}\right)^n \binom{2 
 n}{n} \frac{1}{(4 n+5)^2} + 
 \frac{1}{2} \sum_{n=0}^{\infty} \left(\frac{1}{4}\right)^n \binom{2 n}{n}
 \frac{O_{2 n}}{4 n+3} - \\
 & \frac{3}{2} \sum_{n=0}^{\infty} \left(\frac{1}{4}\right)^n \binom{2 n}{n} \frac{O_{2 n}}{4 n+5}. 
\end{align*}
 From the lemniscate-like constant evaluations in \eqref{directalmostpoised} and \eqref{Slovaca2}, we obtain that 
 \begin{align*} 
 & \int_0^1 \frac{\sqrt{1-x^2} \ln (x)}{\left(2-x^2\right)^{3/2}} \, dx = \\
 & -\frac{\Gamma^2 \left(\frac{1}{4}\right)}{32 \sqrt{2 \pi }} + 
 \frac{(4-\pi ) \Gamma^2 \left(\frac{3}{4}\right)}{8 \sqrt{2 \pi }} + 
 \frac{\pi ^{3/2} (3 \ln (2)+2)}{4 \sqrt{2} \Gamma^2 \left(\frac{1}{4}\right) } + 
 \frac{\left(-1-\frac{\ln (2)}{2}\right) \sqrt{2 \pi ^3}}{\Gamma^2 \left(\frac{1}{4}\right)} + \\ 
 & \frac{13+12 \ln (2)}{8} \sum_{n = 
 0}^{\infty} \left(\frac{1}{4}\right)^n \binom{2 n}{n}
 \frac{1}{4 n+5} - \frac{3}{2} \sum _{n=0}^{\infty} \left(\frac{1}{4}\right)^n \binom{2 n}{n}
 \frac{1}{(4 n+5)^2} - \\
 & \frac{3}{2} \sum _{n = 0}^{\infty} \left(\frac{1}{4}\right)^n \binom{2 n}{n} \frac{O_{2 n}}{4 n+5}. 
\end{align*}
 Applying reindexing arguments, we obtain that 
\begin{align*}
 & \int_0^1 \frac{\sqrt{1-x^2} \ln (x)}{\left(2-x^2\right)^{3/2}} \, dx = \\
 & \frac{-\Gamma^4 \left(\frac{1}{4}\right)-8 \pi ^2 (2+\pi +\ln (2))}{32 \sqrt{2 \pi } \Gamma \left(\frac{1}{4}\right)^2} + \\ 
 & \frac{3+4 \ln (2)}{8} \sum_{n = 
 0}^{\infty} \left(\frac{1}{4}\right)^n \binom{2 n}{n} 
 \frac{1}{2 n-1} + \left(\frac{1}{2}+\frac{\ln (2)}{2}\right) \sum _{n=0}^{\infty} 
 \left(\frac{1}{4}\right)^n \binom{2 n}{n} 
 \frac{1}{4 n+1} - \\ 
 & \frac{1}{2} \sum_{n = 0}^{\infty} \left(\frac{1}{4}\right)^n \binom{2 n}{n} 
 \frac{1}{ (4 n + 
 1)^2} + \frac{9}{8} \sum_{n = 
 0}^{\infty} \left(\frac{1}{4}\right)^n \binom{2 n}{n} \frac{1}{ 4 n - 3} - \\ 
 & \frac{3}{4} \sum _{n=0}^{\infty} \left(\frac{1}{4}\right)^n \binom{2 n}{n} \frac{1}{ (4 n-1)} - \frac{1}{2} 
 \sum _{n=0}^{\infty} \left(\frac{1}{4}\right)^n \binom{2 n}{n}
 \frac{O_{2 n}}{2 n-1} - \\ 
 & \frac{1}{2} \sum_{n = 
 0}^{\infty} \left(\frac{1}{4}\right)^n \binom{2 n}{n} \frac{ O_{2 n}}{4 n+1}. 
\end{align*} 
 From an equivalent formulation of the generating function for the sequence of Catalan numbers, we obtain that: 
\begin{align*}
 & \int_0^1 \frac{\sqrt{1-x^2} \ln (x)}{\left(2-x^2\right)^{3/2}} \, dx = \\ 
 & \frac{-\Gamma^4 \left(\frac{1}{4}\right) - 
 8 \pi ^2 (2+\pi +\ln (2))}{32 \sqrt{2 \pi } \Gamma^2 \left(\frac{1}{4}\right)} + \\ 
 & \left(\frac{1}{2}+\frac{\ln(2)}{2}\right) \sum _{n=0}^{\infty} \left(\frac{1}{4}\right)^n \binom{2 n}{n} 
 \frac{1}{4 n+1} - \frac{1}{2} \sum _{n=0}^{\infty} 
 \left( \frac{1}{4}\right)^n \binom{2 n}{n} \frac{1}{ (4 n + 1)^2} + \\
 & \frac{9}{8} \sum _{n=0}^{\infty} \left(\frac{1}{4}\right)^n \binom{2 n}{n} 
 \frac{1}{ 4 n -3 } +\sum _{n=0}^{\infty} -\frac{3}{4} \left(\frac{1}{4}\right)^n \binom{2 n}{n}
 \frac{ 1}{ 4 n - 1} - \\
 & \frac{1}{2} \sum_{n = 
 0}^{\infty} \left(\frac{1}{4}\right)^n \binom{2 n}{n} 
 \frac{ O_{2 n}}{2 n-1} - \frac{1}{2} \sum_{n = 
 0}^{\infty} \left(\frac{1}{4}\right)^n \binom{2 n}{n} \frac{ O_{2 n}}{4 n+1}. 
\end{align*}
 From the classical lemniscate constant evaluation shown in \eqref{lemniscateA}, we obtain that 
\begin{align*}
 & \int_0^1 \frac{\sqrt{1-x^2} \ln (x)}{\left(2-x^2\right)^{3/2}} \, dx = \\ 
 & \frac{-\Gamma^4 \left(\frac{1}{4}\right) 
 - 8 \pi ^2 (2+\pi +\ln (2))}{32 \sqrt{2 \pi } 
 \Gamma^2 \left(\frac{1}{4}\right)}+\frac{\left(\frac{1}{2}+\frac{\ln (2)}{2}\right) \sqrt{\pi } 
 \Gamma \left(\frac{5}{4}\right) }{\Gamma \left(\frac{3}{4}\right)} - \\ 
 & \frac{1}{2} \sum_{n=0}^{\infty} \left(\frac{1}{4}\right)^n \binom{2 n}{n} \frac{1}{ (4 n + 1)^2} + 
 \frac{9}{8} \sum_{n = 
 0}^{\infty} \left(\frac{1}{4}\right)^n \binom{2 n}{n} \frac{1}{ 4 n-3} - \\ 
 & \frac{3}{4} \sum_{n = 
 0}^{\infty} \left(\frac{1}{4}\right)^n
 \binom{2 n}{n}
 \frac{ 1 }{ (4 n-1)}-\frac{1}{2} \sum_{n = 
 0}^{\infty} \left(\frac{1}{4}\right)^n \binom{2 n}{n}
 \frac{ O_{2 n}}{2 n-1} - \\
 & \frac{1}{2} \sum_{n=0}^{\infty} \left(\frac{1}{4}\right)^n \binom{2 n}{n} \frac{O_{2 n}}{4 n+1}. 
\end{align*}
 From Dixon's formula in \eqref{Dixonmain}, we obtain that 
\begin{align*}
 & \int_0^1 \frac{\sqrt{1-x^2} \ln (x)}{\left(2-x^2\right)^{3/2}} \, dx = \\ 
 & \frac{2 \pi \Gamma \left(-\frac{1}{4}\right) (2+\pi +\ln (2))+\sqrt{2} \Gamma \left(\frac{1}{4}\right)^3 (3-\pi + 4 \ln (2))}{64 \sqrt{\pi } \Gamma
 \left(\frac{1}{4}\right)}+ \\ 
 & \frac{9}{8} \sum_{n = 
 0}^{\infty} \left(\frac{1}{4}\right)^n \binom{2 n}{n} \frac{1}{ 4 n - 3} -\frac{3}{4} 
 \sum_{n = 0}^{\infty} \left(\frac{1}{4}\right)^n \binom{2 n}{n}
 \frac{ 1 }{ 4 n - 1} - \\ 
 & \frac{1}{2} \sum_{n = 
 0}^{\infty} \left(\frac{1}{4}\right)^n \binom{2 n}{n} 
 \frac{ O_{2 n}}{2 n-1} - 
 \frac{1}{2} \sum _{n=0}^{\infty} \left(\frac{1}{4}\right)^n \binom{2 n}{n}
 \frac{O_{2 n}}{4 n+1}. 
\end{align*}
 From the lemniscate-like constant evaluation in \eqref{Slovaca1}, we obtain that 
\begin{align*}
 & \int_0^1 \frac{\sqrt{1-x^2} \ln (x)}{\left(2-x^2\right)^{3/2}} \, dx = \\
 & \frac{\sqrt{2} \Gamma \left(\frac{1}{4}\right)^3 (3-\pi +\ln (2))+2 \pi \Gamma \left(-\frac{1}{4}\right) (2+\pi +\ln (2))}{64 \sqrt{\pi } \Gamma
 \left(\frac{1}{4}\right)} + \\ 
 & \frac{9}{8} \sum_{n = 
 0}^{\infty} \left(\frac{1}{4}\right)^n \binom{2 n}{n} \frac{1}{ 4 n -3} - \frac{3}{4} 
 \sum_{n=0}^{\infty} \left(\frac{1}{4}\right)^n \binom{2 n}{n}
 \frac{ 1}{ 4 n - 1} - \\ 
 & \frac{1}{2} \sum _{n=0}^{\infty} \left(\frac{1}{4}\right)^n \binom{2 n}{n}
 \frac{ O_{2 n}}{2 n-1}. 
\end{align*}
 Using the moment formula $$ O_{2 n}=\int_0^1 \frac{1-x^{4 n}}{1-x^2} \, dx, $$ together with the power series identity $$ \sum_{n = 0}^{\infty } 
 \left(\frac{1}{4}\right)^n \binom{2 n}{n} \frac{ 1-x^{4 n} }{(2 n-1) \left(1-x^2\right)} = -\frac{\sqrt{1-x^4}}{x^2-1}, $$ we may oobtain that $$ 
 \int_{0}^{1} -\frac{\sqrt{1-x^4}}{x^2-1} \, dx = \text{{\bf E}}(i), $$ so that we may apply a known elliptic integral singular value to evaluate the 
 remaining harmonic sum. Explicitly, 
\begin{align*}
 & \int_0^1 \frac{\sqrt{1-x^2} \ln (x)}{\left(2-x^2\right)^{3/2}} \, dx = \\ 
 & \frac{-48 \pi ^2-8 \pi ^3-8 \pi ^2 \ln (2)+\Gamma^4 \left(\frac{1}{4}\right) (-1 - \pi +\ln (2))}{32 \sqrt{2 \pi } \Gamma^2 \left(\frac{1}{4}\right)} + \\ 
 & \frac{9}{8} \sum_{n = 0}^{\infty} \left(\frac{1}{4}\right)^n \binom{2 n}{n} \frac{1}{ 4 n-3} - \frac{3}{4} \sum _{n=0}^{\infty } 
 \left(\frac{1}{4}\right)^n \binom{2 n}{n} \frac{1 }{4 n - 1}. 
\end{align*}
 Applying reindexing arguments, we obtain that 
\begin{align*}
 & \int_0^1 \frac{\sqrt{1-x^2} \ln (x)}{\left(2-x^2\right)^{3/2}} \, dx = \\ 
 & \frac{3}{8}-\frac{\pi ^{3/2} (6+\pi +\ln (2))}{4 \sqrt{2} \Gamma^2 \left(\frac{1}{4}\right)} + 
 \frac{\Gamma^2 \left(\frac{1}{4}\right) (-1-\pi +\ln
 (2))}{32 \sqrt{2 \pi }} - \\
 & \frac{3}{16} \sum_{n = 
 0}^{\infty} \left(\frac{1}{4}\right)^n \binom{2 n}{n}
 \frac{1}{n + 1} + 
 \frac{3}{8} \sum _{n=0}^{\infty}
 \left(\frac{1}{4}\right)^n \binom{2 
 n}{n} \frac{1}{4 n + 1} + \\
 & \frac{3}{4} \sum _{n = 0}^{\infty} \left(\frac{1}{4}\right)^n \binom{2 n}{n} \frac{1}{4 n+3}
\end{align*}
 Using the generating function for the sequence of Catalan numbers, together with the evaluations for the classical lemniscate constants, we obtain $$ 
 \int_0^1 \frac{\sqrt{1-x^2} \ln (x)}{\left(2-x^2\right)^{3/2}} \, dx = \frac{\Gamma^2 \left(\frac{1}{4}\right) (4-2 \pi +2 \ln (2))}{64 \sqrt{2 \pi }} - 
 \frac{\pi ^{3/2} (\pi +\ln (2))}{4 \sqrt{2} \Gamma^2 \left(\frac{1}{4}\right)}. $$ So, from \eqref{oneintremains}, we find that the integral in 
 \eqref{inttosplit} may be reduced to the following: $$ \int_0^1 \frac{\sqrt{1-x^2} \left(x^2-4\right) \ln (x)}{\left(2-x^2\right)^{3/2}} \, dx 
 = \frac{\sqrt{2} \pi ^{3/2}}{\Gamma^2 \left(\frac{1}{4}\right)}+\left(\frac{\sqrt{\frac{\pi }{2}}}{16}-\frac{\ln (2)}{16 \sqrt{2 \pi }}\right) \Gamma^2 
 \left(\frac{1}{4}\right). $$
 So, from the equality of \eqref{maindesired} and \eqref{bygenbin}, 
 we obtain the equality 
 $$ \sum _{n=0}^{\infty} \left(\frac{1}{32}\right)^n \left(H_{2 n}-H_n\right) \binom{2 n}{n}^2
 = \frac{(5 \ln (2)-\pi ) \Gamma^2 \left(\frac{1}{4}\right)}{8 \pi ^{3/2}}. $$ 
 So, from the harmonic sum evaluation in \eqref{mainTauraso}
 derived from Bailey's theorem, we obtain that: 
\begin{equation}\label{fromBaileyend}
 \sum _{k=0}^{\infty} \left(\frac{1}{32}\right)^k \binom{2 k}{k}^2 H_{2 k} 
 = \frac{(\pi -3 \ln (2)) \Gamma^2 \left(\frac{1}{4}\right)}{8 \pi ^{3/2}}. 
\end{equation}
 So, by taking an appropriate linear combination of  \eqref{fromBaileyend} and the formula in  \eqref{mainTauraso} derived from Bailey's theorem, 
 we obtain the desired result. 
\end{proof}

 The formulas $$ \sum_{k=0}^{\infty} \left( -\frac{1}{16} \right)^{k} \binom{2k}{k}^2 H_{k} = \frac{\Gamma^{2}\left( \frac{1}{4} \right)}{4 
 \sqrt{2 \pi^3}} \left( \pi - 5 \ln(2) \right)$$ and $$ \sum_{k=0}^{\infty} \left( -\frac{1}{16} \right)^{k} \binom{2k}{k}^2 H_{2k} = 
 \frac{ \Gamma^{2}\left( \frac{1}{4} \right) }{ 8 \sqrt{2 \pi^3} } \left( \pi - 6 \ln(2) \right) $$ were highlighted as part of Theorem 4 in 
 \cite{ChuCampbell2021JMAA} and proved 
 via a linearization method introduced in \cite{ChuCampbell2021JMAA}. 
 So, by taking an appropriate linear combination of the above formulas, 
 this gives us a proof of Sun's conjectured formula in \eqref{Sun1}. 

 The new formula in \eqref{fromBaileyend} is of interest in its own right and recalls the open problem recently considered in \cite{Chu2022} to evaluate 
 $$ \sum_{k=0}^{\infty} \left( \frac{1}{32} \right)^{k} \binom{2k}{k}^{2} H_{k}^{(2)} 
 \ \ \ \text{and} \ \ \ \sum_{k=0}^{\infty} \left( \frac{1}{32} \right)^{k} \binom{2k}{k}^{2} H_{k}^{2} $$ 
 symbolically, writing $H_{m}^{(2)} = 1 + \frac{1}{2^2} + \cdots + \frac{1}{m^2}$. 
 Our evaluation in \eqref{fromBaileyend}
 also motivates our interest in the open problem of evaluating 
 $$ \sum _{k=0}^{\infty} \left(\frac{1}{32}\right)^k \binom{2 k}{k}^2 H_{2 k}^{(2)} 
 \ \ \ \text{and} \ \ \ \sum _{k=0}^{\infty} \left(\frac{1}{32}\right)^k \binom{2 k}{k}^2 H_{2 k}^{2} $$
 symbolically.

 \

John M.\ Campbell

Department of Mathematics

Toronto Metropolitan University

{\tt jmaxwellcampbell@gmail.com}

\end{document}